\documentclass[11pt,leqno]{article}
\usepackage{amssymb, amscd, amsmath, amsthm}

\allowdisplaybreaks

\def\ve{\varepsilon}

\def\mod{\,\text{\rm mod}\;}
\def\beq{\begin{equation}}
\def\eeq{\end{equation}}
\def\cite#1{{\rm [#1]}}

\newtheorem{lemma}{Lemma}

\newtheorem{cor}{Corollary}
\theoremstyle{definition}

\newtheorem*{rema}{Remark}

\begin{document}

\numberwithin{equation}{section}
\title{Patterns of primes}

\author{J\'anos Pintz\thanks{Supported by OTKA Grants K72731, K67676 and ERC-AdG.228005.}}

\date{}
\maketitle

\section{Introduction}
\label{sec:1}
A few years ago Green and Tao \cite{GT} proved their striking result about patterns in primes.

\medskip
\noindent
{\bf Theorem {\rm (Green--Tao)}.}
{\it The primes contain arbitrarily long arithmetic progressions.}

\medskip
The method of proof immediately gave that the same result is true for any subset $\mathcal P'$ of the primes $\mathcal P = \{p_n\}^\infty_{n = 1}$ with positive relative upper density, that is with
\beq
\limsup_{N \to \infty} \frac{\bigl|\mathcal P' \cap [1, N]\bigr|}{\pi(N)} > 0
\label{eq:1.1}
\eeq
(where $\pi(N)$ denotes the number of primes less or equal to~$N$, $|\mathcal A|$ the number of elements of a set $A$, and the fact that the number of $m$-term arithmetic progressions obtained below $N$ is $\gg N^2(\log N)^{-m}$.

Another, albeit conditional result of Goldston, Y{\i}ld{\i}r{\i}m and the author \cite{GPY2} yielded the existence of other patterns.

\medskip
\noindent
{\bf Theorem {\rm (\cite{GPY2})}.}
{\it If the primes have a distribution level $\vartheta > 1/2$, that is, if for any positive $\ve$ and $A$ we have
\beq
\sum_{q \leq N^{\vartheta - \ve}} \max_{\substack{a\\ (a,q) = 1}} \biggl| \sum_{\substack{p \leq x\\ p\equiv a(\mod q)}} \log p - \frac{N}{\varphi(q)} \biggr| \ll_{\ve, A} \frac{N}{\log^A N},
\label{eq:1.2}
\eeq
then there exists a positive even $d \leq C_1(\vartheta)$ and infinitely many pairs of primes}
\beq
p, p + d \in \mathcal P.
\label{eq:1.3}
\eeq

\medskip
The author showed recently that a combination of the two above results is possible, showing thereby new patterns of primes.

\medskip
\noindent
{\bf Theorem {\rm \cite{Pin}}.}
{\it If the primes have a distribution level $\vartheta > 1/2$ then there exists a positive even $d \leq C_1(\vartheta)$ such that the set $\mathcal P(d)$ of primes $p$ satisfying \eqref{eq:1.3} contains arbitrarily long arithmetic progressions.}

\begin{rema}
The above (conditionally existing) patterns form two-dimensional arithmetic progressions with one difference being bounded.
\end{rema}

\begin{rema}
In the above two theorems we have $0 < d \leq 16$ if $\vartheta > 0.971$, in particular, if the Elliott--Halberstam conjecture \cite{EH} $\vartheta = 1$ is true.
On the other hand, the best unconditional result $\vartheta = 1/2$, the celebrated Bombieri--Vinogradov theorem, unfortunately does not imply the existence of infinitely many bounded gaps between consecutive primes.
\end{rema}

However, beside Selberg's sieve, the Bombieri--Vinogradov theorem played a crucial role in the proof \cite{GPY2} of
\beq
\Delta_1 = 0, \quad \text{ where } \ \Delta_\nu = \liminf_{n \to \infty} (p_{n + \nu} - p_n) / \log p_n,
\label{eq:1.4}
\eeq
thereby improving the best known bound
\beq
\Delta_1 < 0.2486
\label{eq:1.5}
\eeq
of Helmut Maier \cite{Mai}.

One important question which remained open after the work \cite{GPY1} was whether the small gaps of size $< \eta \log p$ appear with a positive density for any $\eta > 0$.
Since the existence of some patterns in a subset of primes can be deduced from information about the relative density of the subset, this gives an extra interest to problems asking whether some ``events'' as short gaps between consecutive primes or short blocks of gaps between consecutive primes appear in a positive proportion of all cases or not.
This motivates the definition of the quantities
\beq
\Delta^*_\nu = \inf \Biggl\{ c_\nu; \ \liminf_{x \to \infty} \frac{\bigl|\{p_n \leq x; p_{n + \nu} - p_n \leq c_\nu \log p_n\}\bigr|}{\pi(x)} > 0\Biggr\}.
\label{eq:1.6}
\eeq

The methods of Hardy--Littlewood, \cite{HL, Ran} Erd\H{o}s \cite{Erd}, Bombieri--Davenport \cite{BD} and its refinements by Huxley \cite{Hux1, Hux2, Hux3} yielded always a positive proportion of small gaps; however, the ingenious improvement \eqref{eq:1.5} by H. Maier \cite{Mai} just showed the existence of a rare set of small gaps or blocks of gaps.
Thus, our knowledge in the time of Maier's work was as follows:
\begin{alignat}2
\label{eq:1.7}
\Delta^*_1 &\leq 0.4425... \ \text{\cite{Hux2}}, \ &
\Delta_1 &\leq e^{-\gamma} \cdot 0.4425... = 0.2486...  \ \text{\cite{Mai}},\\
\label{eq:1.8}
\Delta^*_\nu &\leq \nu - \frac58 + o(1) \ \text{\cite{Hux2}}, \  &
\Delta_\nu &\leq e^{-\gamma} \left(\nu - \frac58 + o(1)\right) \ \text{\cite{Mai}}.
\end{alignat}

\begin{rema}
A slight improvement over \eqref{eq:1.7}--\eqref{eq:1.8} is contained in \cite{Hux3}.
However, Maier refined the version \eqref{eq:1.7}--\eqref{eq:1.8} of \cite{Hux2}.
\end{rema}

Goldston and Y{\i}ld{\i}r{\i}m \cite{GY} worked out a method which yielded
\beq
\Delta^*_1 \leq 1/4
\label{eq:1.9}
\eeq
and it remained unclear whether the method of \cite{GPY2} proving $\Delta_1 = 0$ is able to yield $\Delta^*_1 = 0$ too.
Very recently, this question was answered positively.

\medskip
\noindent
{\bf Theorem {\rm \cite{GPY3}}.}
{\it Unconditionally we have $\Delta^*_1 = 0$; further the Elliott--Halberstam conjecture \cite{EH} implies $\Delta^*_2 = 0$.}

\medskip
Taking into account the stronger form of the Green--Tao Theorem (cf.\ \eqref{eq:1.1}) the above theorem implies

\begin{cor}
\label{cor:1}
Let $\eta > 0$ be arbitrary, $p'$ be the prime following $p$.
Then the set
\beq
\mathcal P'(\eta) = \{p \in \mathcal P;\ p' - p \leq \eta \log p\}
\label{eq:1.10}
\eeq
contains arbitrarily long arithmetic progressions.
The same is true under EH for
\beq
\mathcal P''(\eta) = \bigl\{p \in \mathcal P;\ p_{n + 2} - p_n \leq \eta \log p_n\bigr\}.
\label{eq:1.11}
\eeq
\end{cor}

The method of proof of \cite{GPY3} can also yield that the best unconditional bound of \cite{GPY2},
\beq
\Delta_\nu \leq \bigl(\sqrt{\nu} - \sqrt{2\vartheta}\bigr)^2,
\label{eq:1.12}
\eeq
can be refined to
\beq
\Delta^*_\nu \leq \bigl(\sqrt{\nu} - \sqrt{2\vartheta}\bigr)^2.
\label{eq:1.13}
\eeq

\begin{rema}
The unconditional result
\beq
\Delta_\nu \leq e^{-\gamma}\bigl(\sqrt{\nu} - 1\bigr)^2
\label{eq:1.14}
\eeq
of the work \cite{GPY3} cannot be modified to yield the same estimate to $\Delta^*_\nu$ as well, since it uses Maier's matrix method too (as can be guessed from the factor $e^{-\gamma})$, which in general yields just a negligible portion of primes with a given property.
\end{rema}

The aim of this note is to show that the method of the mentioned work \cite{GPY3} can be modified to yield for any fixed $\eta > 0$ for $N \to \infty$ many $\nu + 1$-dimensional patterns of type $(d_1, \dots, d_\nu)$ with $0 < d_1 < \dots < d_\nu$ such that
\begin{align}
\label{eq:1.15}
\bigl|\mathcal P(d_1, \dots, d_\nu)\bigr|
&= \bigl| \bigl\{p \in \mathcal P, \, p \in [N, 2N),\, p + d_i \in \mathcal P \, (i\! =\! 1, ..., \nu)\bigr \}\bigr| \\
&\geq \frac{c_1(\nu, \eta)N}{(\log N)^{\nu + 1}},
\nonumber
\end{align}
\beq
d_\nu \leq \bigl(\Delta^*_\nu + \eta \bigr) \log N,
\label{eq:1.16}
\eeq
where we choose $c_1(\nu, \eta)$ sufficiently small, depending on $\eta$ and~$\nu$.

The exact formulation of our result to be proved is as follows.

\medskip
\noindent
{\bf Theorem.}
{\it Let $\eta > 0$ be any positive constant, $\nu$ and $m$ natural numbers.
Then we have a positive constant $c(\eta, \nu)$ depending on $\eta$ and~$\nu$ such that for any $N > N_0(\eta, \nu, m)$ we have a set $\mathcal D^\nu_{\!N}$ of $\nu$-tuples $(d_1, \dots, d_\nu)$ with $0 < d_1 < \dots < d_\nu$ such that
\beq
\bigl|\mathcal D^\nu_{\!N}\bigr| \geq c(\eta, \nu) \log^\nu\! N
\label{eq:1.17}
\eeq
and every element of $\mathcal D^\nu_{\!N}$ satisfies \eqref{eq:1.15} and \eqref{eq:1.16}.}

\medskip
\noindent
{\bf Corollary}
{\it Under the above conditions, if $\bigl(d_i\bigr)^\nu_{i = 1} \in \mathcal D^\nu_{\!N}$ then the set\break
$\mathcal P(d_1, \dots, d_\nu)$ of primes contains at least $c'(\eta, \nu, m) \frac{N^2}{\log^m\! N}$ arithmetic progressions of length~$m$.}

\begin{rema}
In such a way we actually obtain a large number of $\nu + 1$-dimensional arithmetic progressions, more exactly a positive proportion of all $\nu$-tuples $(d_1, \dots, d_\nu)$ with $0 < d_1 < \dots < d_\nu \leq (\Delta^*_\nu + \eta) \log N$ will appear as a configuration of primes $p^{(j)} + d_i \in \mathcal P$, $p^{(j)} \in \mathcal P$ where $\{p^{(j)}\}^m_{j = 1}$ forms an $m$-term arithmetic progressions (and consequently so do the primes $p^{(j)} + d_i$ for all $i \in [1, \nu]$).
\end{rema}

\section{Proof of the Theorem}
\label{sec:2}

The number of $\nu + 1$-tuple of primes satisfying $p + d_i \in \mathcal P$ for a concrete $\mathcal D = (d_1, \dots, d_\nu)$ can be estimated from above by Selberg's sieve (cf.\ Theorem~5.1 of \cite{HR} or Theorem~2 in \S~2.2.2 of \cite{Gre})
\beq
\bigl|\mathcal P(d_1, \dots, d_\nu)\bigr| \ll_\nu \frac{N\mathfrak S(\mathcal D^+)}{(\log N)^{\nu + 1}}, \quad \mathcal D^+ = \mathcal D \cup \{0\}.
\label{eq:2.1}
\eeq
This would be immediately sufficient to prove a positive proportion of the required prime $\nu + 1$-tuples if $\mathfrak S(\mathcal H)$ were bounded for $k$-tuples $\mathcal H$ of a given size, which is not the case.
However, using the definition \eqref{eq:1.6} of $\Delta^*_\nu$, with the notation
\beq
H := \bigl\lfloor (\Delta^*_\nu + \eta) \log N \bigr\rfloor,
\label{eq:2.2}
\eeq
we have (with $c_i$ depending always on $\eta$ and $\nu$) by the definition of $\Delta^*_\nu$ and \eqref{eq:2.1}
\beq
\frac{c_1 N}{\log N} \leq \sum_{\substack{\mathcal D \subset [1, H]\\
|\mathcal D| = \nu}} \frac{c_2 N \mathfrak S(\mathcal D^+)}{(\log N)^{\nu + 1}} \leq c_3 \frac{N}{\log N} \cdot \frac1{H^\nu} \sum_{\substack{\mathcal D \subset [1, H]\\
|\mathcal D| = \nu}} \mathfrak S(\mathcal D^+).
\label{eq:2.3}
\eeq
Deleting from the summation those $\mathcal D$'s for which with a sufficiently small $c_4$ we have
\beq
\bigl| \mathcal P(d_1, \dots, d_\nu)\bigr| \leq \frac{c_4 N \mathfrak S(\mathcal D^+)}{(\log N)^{\nu + 1}},
\label{eq:2.4}
\eeq
we obtain for a subset $\mathbb D$ of all $\mathcal D \subset [1,H]$ (we denote summation over this subset by $\sum^*$)
\beq
\frac{c_1 N}{2 \log N} \leq c_3 \frac{N}{\log N} \cdot \frac1{H^\nu} \underset{\mathcal D \in \mathbb D}{\sum\nolimits^*} \mathfrak S(\mathcal D^+).
\label{eq:2.4masodik}
\eeq
In order to prove our theorem it is clearly sufficient to show
\beq
\sum_{\mathcal D \in \mathbb D} 1 \geq c_5 H^\nu .
\label{eq:2.5}
\eeq
Now, using Cauchy's inequality, \eqref{eq:2.4masodik} implies
\beq
H^\nu \leq c_6 \underset{\mathcal D \in \mathbb D}{\sum\nolimits^*} \mathfrak S(\mathcal D^+) \leq c_6 \Biggl(\sum_{\mathcal D \in \mathbb D} 1 \sum_{\substack{\mathcal D \subset [1, H]\\ |\mathcal D| = \nu}} \mathfrak S^2 (\mathcal D^+)\Biggr)^{1/2}.
\label{eq:2.6}
\eeq
Hence, in order to show \eqref{eq:2.5}, thereby our Theorem, it is sufficient to show the following

\begin{lemma}
\label{lem:1}
For fixed $\nu$ and any $H > H_0(\nu)$ we have
\beq
\sum_{\substack{\mathcal D \subset [1, H]\\ |\mathcal D| = \nu}} \mathfrak S^2(\mathcal D^+) \leq c_7(\nu) H^\nu.
\label{eq:2.7}
\eeq
\end{lemma}

\begin{rema}
The parameter $H$ can be arbitrary here, not just that given in \eqref{eq:2.2}.
\end{rema}

\begin{rema}
The above lemma is somewhat analogous to Gallagher's theorem
\beq
\sum_{\substack{\mathcal D \subset [1, H]\\ |\mathcal D| = \nu}} \mathfrak S(\mathcal D) \sim H^\nu,
\label{eq:2.8}
\eeq
the difference being the non-essential appearance of $\mathcal D^+ = \mathcal D \cup \{0\}$ in place of $\mathcal D$ and the more essential change in the exponent: two instead of one.
\end{rema}

Since the singular series is interesting in itself and appears often in additive number theory, it might be interesting to prove with the same effort a more general form of it as

\begin{lemma}
\label{lem:2}
For fixed $\nu$ $r$ and $H > H_0(\nu, r)$ we have
\beq
S(\nu, r) = \sum_{\mathcal D \subset [1, H]}
\mathfrak S^2(\mathcal D^+) \leq c_8(\nu,r) H^\nu.
\label{eq:2.9}
\eeq
\end{lemma}

\begin{rema}
The condition $H > H_0(\nu, r)$ and $H > H_0(\nu)$ is naturally not necessary if we do not care about the values of the constants $c_7(\nu)$ and $c_8(\nu, r)$.
\end{rema}

\begin{rema}
In case of $r = 1$ we will additionally show, similarly to \eqref{eq:2.8}, $S(\nu, r) \sim H^\nu$ as $H \to \infty$.
This slightly modified form implies easily the original Gallagher's theorem too, by dividing all possible $\nu + 1$-tuples according to the smallest element of it and using that $\mathfrak S(\mathcal H)$ is invariant under translation.
\end{rema}

\begin{proof}[Proof of Lemma~\ref{lem:2}]
We will prove in fact a little bit more.
Namely, the fact that extending every concrete admissible  $\mathcal D \cup \{0\}$ of size $t + 1 \geq 1$ with just one element running over $[1, H]$ the square of the singular series will be larger at most by a factor depending on~$t$.
In such a way, \eqref{eq:2.9} follows by induction from
\beq
S^*(t, r, \mathcal D) := \sum_{1 \leq h \leq H,\ h \notin \mathcal D} \biggl(\frac{\mathfrak S(\mathcal D^+ \cup \{h\})}{\mathfrak S(\mathcal D^+)} \biggr)^r \ll  H
\label{eq:2.10}
\eeq
where $\mathcal D^+$ is any admissible set of size $t + 1$ and, as in the following, we will not mark the dependence of the constants implied by $\ll$ or $0$ symbols on $t$ and~$r$.
We can start with $\mathcal D^+ = \mathcal D \cup \{0\} = \{0\}$, that is, with the case $t = 0$.

In order to investigate \eqref{eq:2.10} we study the ratio in \eqref{eq:2.10} for any single $h$ and denote
\beq
\nu'_p\! =\! \nu_p(\mathcal D^+ \cup \{h\}), \ \, \nu_p = \nu_p(\mathcal D^+), \ \, y = \frac{\log H}{2}, \ \, P = \prod_{p \leq y} p, \ \, \Delta := \prod^\nu_{i = 1} (h - d_i).
\label{eq:2.11}
\eeq
With these notations we can write
\beq
\frac{\mathfrak S(\mathcal D^+ \cup \{h\})}{\mathfrak S(\mathcal D^+)} = \prod_p \frac{1 - \frac{\nu'_p}{p}}{\left(1 - \frac{\nu_p}{p}\right) \left(1 - \frac1p\right)} = \underset{p\leq y}{\prod\nolimits_1} \cdot
\underset{\scriptstyle p > y \atop\scriptstyle p \,| \Delta}{\prod\nolimits_2} \cdot
\underset{\scriptstyle p > y \atop\scriptstyle p \nmid \Delta}{\prod\nolimits_3}.
\label{eq:2.12}
\eeq
For $p \nmid \Delta$ we have $\nu'_p = \nu_p + 1$, otherwise $\nu'_p = \nu_p$, hence
\beq
\prod\nolimits_3 = \prod_{p > y} \left(1 + O \left(\frac1{p^2}\right)\right) = 1 + O\Bigl(\frac1y\Bigr),
\label{eq:2.13}
\eeq
\beq
\log \prod\nolimits_2 \ll \sum_{p|\Delta,\ p > y} \frac1p \ll \sum_{p|\Delta} \frac{\log p}{y \log y} \ll \frac{\log \Delta}{y \log y} \ll \frac1{\log y}.
\label{eq:2.14}
\eeq

If $H = RP + r$, $0 \leq r < P$ then $\prod_1(h)$ is periodic with period~$P$.
For any $p \leq y$ we have exactly $\nu_p$ possibilities for $h$ with $\nu'_p = \nu_p \mod p$, and $p - \nu_p$ possibilities with $\nu'_p = \nu_p + 1$.
Consequently
\begin{align}
\label{eq:2.15}
\frac1P \sum^P_{h = 1} \prod\nolimits_1 (h)
&= \prod_{p | P} \frac{\left\{\frac{\nu_p}{p} \left(1 - \frac{\nu_p}{p}\right)^r + \left(1 - \frac{\nu_p}{p}\right) \left(1 - \frac{\nu_p + 1}{p}\right)^r \right\}}{\left(1 - \frac{\nu_p}{p}\right)^r\left(1 - \frac1p\right)^r} \\
&= \prod_{p | P} \frac{\frac{\nu_p}{p} + 1 - \frac{\nu_p}{p} - \frac{r(\nu_p + 1)}{p} + O \left(\frac1{p^2}\right)}{1 - \frac{r(\nu_p + 1)}{p} + O \left(\frac1{p^2}\right)}\notag\\
&= \prod_{p|P}\left(1 + O\left(\frac1{p^2}\right)\right) = O(1).
\notag
\end{align}

\eqref{eq:2.13}--\eqref{eq:2.15} together prove the lemma, while for $r = 1$, in order to obtain $\sim$ instead of $\ll$, it is enough to observe that the numerator after the product sign equals exactly $1$ for each prime $p$, and the contribution of the incomplete period, the interval $[RP + 1, RP + r]$, is $\leq P = 0(H)$ by the prime number theorem, since $y = \log H/2$.
\end{proof}

Hence, as mentioned earlier, Lemma~\ref{lem:2} implies the Theorem by \eqref{eq:2.1}--\eqref{eq:2.6}.

\bigskip
\noindent
{\small J\'anos {\sc Pintz}\\
R\'enyi Mathematical Institute of the Hungarian Academy
of Sciences\\
Budapest\\
Re\'altanoda u. 13--15\\
H-1053 Hungary\\
E-mail: pintz@renyi.hu}

\end{document}